\documentclass[11pt,leqno, a4paper]{amsart}

\usepackage{graphicx}
\usepackage{amssymb}
\usepackage{color}
\usepackage{setspace}
\usepackage{mathtools}
\usepackage{dsfont}
\usepackage{commath}
\usepackage{enumerate}
\usepackage{mathrsfs}
\usepackage{fancyhdr}
\usepackage{amsthm}
\usepackage{scalerel} 
\usepackage{fancyhdr}
\usepackage{bbm}
\usepackage{hyperref}
\usepackage[normalem]{ulem}
\usepackage{tikz}
\usepackage{subcaption}

\usetikzlibrary{decorations.fractals}

\newtheorem{theorem}{Theorem}[section]

\newtheorem{lemma}[theorem]{Lemma}

\newtheorem{corollary}[theorem]{Corollary}
\newtheorem{proposition}[theorem]{Proposition}

\theoremstyle{definition}

\newtheorem{example}[theorem]{Example}

\theoremstyle{remark}
\newtheorem{remark}[theorem]{Remark}

\numberwithin{equation}{section}

\newcommand{\rn}{\mathbb{R}^n}

\newcommand{\cH}{\mathcal{H}}

\newcommand{\R}{\mathbb{R}}

\newcommand{\osc}{\operatorname{osc}}

\renewcommand{\tilde}{\widetilde}

\newcommand{\loc}{\operatorname{loc}}

\newcommand{\Langle}{\left \langle}
\newcommand{\Rangle}{\right \rangle}

\DeclareMathOperator{\divr}{div}

\newcommand{\defeq}{\vcentcolon=}

\def\XXint#1#2#3{{\setbox0=\hbox{$#1{#2#3}{\int}$ }
\vcenter{\hbox{$#2#3$ }}\kern-.6\wd0}}

\title[Harnack inequality for the Finsler $\gamma$-Laplacian]{A Harnack inequality for weak solutions of the Finsler $\gamma$-Laplacian}
\author{Max Goering}
\thanks{Max Planck Institute, Leipzig, DE. \ \ {\it Email}: max.goering@mis.mpg.de}
\thanks{The author was partially supported by FRG DMS-1853993.}
\thanks{Data sharing not applicable to this article as no datasets were generated or analysed during the current study.}

\email{goering@mis.mpg.de}

\begin{document}
\maketitle

\begin{abstract}
We study regularity of the Finsler $\gamma$-Laplacian, a general class of degenerate elliptic PDEs which naturally appear in anisotropic geometric problems. Precisely, given any strictly convex family of $C^{1}$-norms $\{ \rho_{x}\}$ on $\R^{n}$ and $\gamma > 1$, we consider the $W^{1,\gamma}(\Omega)$ solutions of the anisotropic PDE
$$
\displaystyle \int_{\Omega} \Langle \rho_{x}(Du)^{\gamma-1} (D \rho_{x})(Du), D \varphi \Rangle = \int_{\Omega} \vec{F} \cdot D \varphi + f \varphi \qquad \forall \varphi \in W^{1,\gamma^{\prime}}_{0}(\Omega).
$$
Under the mild assumption $|\xi|^{-1} \rho_{x}( \xi) \in [\nu, \Lambda]$ for all $(x,\xi) \in \Omega \times \R^{n}$ and some $0 < \nu \le \Lambda < \infty$ we perform a Moser iteration, verifying that sub- and super-solutions satisfy one-sided $\| \cdot \|_{\infty}$ bounds, which together imply solutions are locally bounded. When $u$ is non-negative this also implies a (weak) Harnack inequality. If $f, \vec{F} \equiv 0$ weak solutions also benefit from a strong maximum principle, and a Liouville-type theorem. 
\end{abstract}

\section{Introduction}

In his celebrated work on minimal surfaces \cite{DG54,DG55,DG58}, De Giorgi framed minimal surfaces as the boundaries of sets of locally finite perimeter. In this setting, he showed that if $\partial E$ is flat near $x$ then in fact $\partial E$ is smooth near $x$. The method is by now well-known and has been adapted to many settings and can roughly be summarized in the following 4 steps: 

\noindent (1a) flatness implies that $\partial E$ is mostly the graph of a Lipschitz function $u$, and (1b) the Lipschitz constant can be arbitrarily small if the boundary is sufficiently flat. 

(2) Since $\partial E$ is a minimal surface, it turns out $u$ almost solves the minimal surface equation $$\displaystyle \divr \left( \frac{ Du}{\sqrt{1+|Du|^{2}}} \right) = 0.$$

(3) When the Lipschitz constant is sufficiently small, this means $u$ is basically harmonic. 

(4) From the regularity theory for harmonic functions one can deduce regularity of $\partial E$ near $x$.

We now motivate the study of the Finsler $\gamma$-Laplacian\footnote{In the literature this is typically called the Finsler $p$-Laplacian, but we wish to reserve $p$ to refer to the $\|\cdot\|_{\ell^{p}}$ and $\| \cdot \|_{L^{p}(\R^{n})}$ norms.}, a PDE which generalizes the Finsler Laplacian in an analogous manner to the way the $p$-Laplacian generalizes the Laplacian. To this end, we formally apply De Giorgi's method to an anisotropic minimal surface. Given a set of locally finite perimeter $E \subset \R^{n+1}$ consider the energy
\begin{equation} \label{e:lpsenergy}
\Phi_{p}(E ; A) \defeq \int_{\partial^{*}E \cap A} \| \nu_{E}\|_{\ell^{p}} d \cH^{n}.
\end{equation}

Simplifying our summary of De Giorgi's method a little further, in lieu of (1a) we assume $\partial^{*}E \cap A$ completely coincides with the graph of some Lipschitz $u : \Omega \to \R$. Then,
\begin{equation} \label{e:lpgenergy}
\Phi_{p}(E;A) = \int_{\Omega} \| (-Du,1)\|_{\ell^{p}}.
\end{equation}
Analogous to (2), by computing the outer-variation of \eqref{e:lpgenergy} one discovers $u$ solves the $\| \cdot \|_{\ell^{p}}$-minimal surface equation
\begin{equation} \label{e:lpminsurf}
\divr \left( \frac{ \partial_{i} u |\partial_{i} u|^{p-2}}{\|(-Du,1)\|_{\ell^{p}}^{p-1}} \right) = 0.
\end{equation}
By (1b) we can assume that $\|Du\|_{L^{\infty}(\Omega)} \ll 1$ and hence Taylor expand the elliptic matrix $ \displaystyle A(x) \defeq \frac{Id}{\|(-Du,1)\|_{\ell^{p}}^{p-1}}$ around $Du = 0$ to find that, in place of (3), formally $u$ almost solves the PDE
\begin{equation} \label{e:pseudop}
\divr \left( \partial_{i} u |\partial_{i} u|^{p-2} \right) = 0,
\end{equation}
which is sometimes called the Pseudo $p$-Laplacian or orthotropic $p$-Laplacian. Hence, to formally complete De Giorgi's method, step (4) would indicate one should try to recover regularity of $\partial E$ from the regularity theory for the pseudo $p$-Laplacian. Unfortunately, the best-known regularity theory for \eqref{e:pseudop} is that solutions are Lipschitz when $n \ge 3$, \cite{demengel2016lipschitz}, and $C^{1,\log}$ for $n = 2$, \cite{lindqvist2018regularity}.

In this paper we propose considering more generally the regularity of solutions to the Finsler $\gamma$-Laplacian, \eqref{e:inhomo}, a general class of PDEs which arise naturally in geometric problems. In the homogeneous setting, this simplifies to fixing some $\gamma > 1$, $\Omega \subset \R^{n}$, and a $C^{1}$, strictly-convex norm $\rho$, then considering the weak solutions of
\begin{equation} \label{e:thesimplepde}
\int_{\Omega} \Langle \rho(Du)^{\gamma-1} (D\rho)(Du), D \varphi \Rangle d \cH^{n}  = 0 \qquad \forall \varphi \in W^{1,\gamma^{\prime}}_{0}(\Omega).
\end{equation}

We note that in this setting, \eqref{e:thesimplepde} is precisely the outer-variation of the functional
\begin{equation} \label{e:functional}
\int_{\Omega} \rho(Du)^{\gamma} d \cH^{n}.
\end{equation}

\begin{example} \label{x:pde1}
Several immediate examples arise when considering \eqref{e:thesimplepde} and \eqref{e:functional}.
\begin{enumerate}
\item[(i)] If $\rho = | \cdot |$ and $\gamma = 2$ we recover the Laplacian/Dirichlet energy.
\item[(ii)] If $\rho = | \cdot |$ and $\gamma = p$ \eqref{e:thesimplepde} recovers the $p$-Laplacian.
\item[(iii)] If $\rho = \| \cdot \|_{\ell^{p}}$ and $\gamma = p$ \eqref{e:thesimplepde} recovers the pseudo $p$-Laplacian.
\item[(iv)] If $\rho(\xi) \defeq \langle A \xi, \xi \rangle^{1/2}$ for some positive definite matrix $A$, and $\gamma = 2$, \eqref{e:thesimplepde} recovers constant coefficient divergence form elliptic PDEs.
\end{enumerate}
\end{example}

In general, given some $\Omega \subset \R^{n}$ we consider a function $\rho : \Omega \times \R^{n} \to [0, \infty)$ so that $\rho(x, \cdot) = \rho_{x}$ is a norm for all $x \in \Omega$. Then, the inhomogeneous Finsler $\gamma$-Laplacian with respect to $\rho$ is given by
\begin{equation} \label{e:inhomo}
\int_{\Omega} \Langle \rho_{x}(Du))^{\gamma-1} (D \rho_{x})(Du) , D \varphi  \Rangle = \int_{\Omega} \Langle \vec{F}, D\varphi \Rangle + f \varphi \quad \forall \varphi \in W^{1,\gamma^{\prime}}(\Omega)
\end{equation}
where
\begin{equation}\label{e:Ff}
\vec{F} \in L^{q}_{\loc}(\Omega) \quad \text{and} \quad f \in L^{\frac{q}{\gamma^{\prime}}}_{\loc}(\Omega) \quad \text{for some } q > \frac{n}{\gamma-1}.
\end{equation}

\begin{example} \label{x:pde2} In the case $\vec{F} = 0$ and $f = 0$,
\begin{enumerate}
\item[(i)] If $A(x) \in \R^{n \times n}$ is uniformly elliptic for $x \in \Omega$, $\rho_{x}(\xi) = \langle A(x) \xi, \xi \rangle^{1/2}$, and $\gamma = 2$ then \eqref{e:inhomo} recovers all divergence form elliptic PDEs with $L^{\infty}$-coefficients.
\item[(ii)] If 
$$
\rho_{x}(\xi) = \frac{\| \xi \|_{\ell^{p}}}{\|(-Du(x),1)\|_{\ell^{p}}^{\frac{1}{p^{\prime}}}},
$$ 
then \eqref{e:inhomo} recovers the $\| \cdot \|_{\ell^{p}}$-minimal surface equation \eqref{e:lpminsurf}.
\end{enumerate}
\end{example}

The first main result is a sup-bound for non-negative weak subsolutions of \eqref{e:inhomo}.

\begin{theorem} Suppose that for some $1 < \gamma < n$, $u \in W^{1,\gamma}(\Omega)$ is a non-negative subsolution of \eqref{e:inhomo}, that $\rho: \Omega \times \rn \setminus \{0\} \to (0,\infty)$ satisfies \eqref{e:meas}, \eqref{e:p1h}, \eqref{e:cx}, and \eqref{e:eb}, and $\vec{F},f$, and $q$ satisfy \eqref{e:Ff}. If $0 < r < R < 1$ and $B_{R} \subset \Omega$, then for all $0 < p < \infty$, there exists $C = C(n,\gamma,\nu,p)$ so that
\begin{equation} \label{e:supbound}
\sup_{B_{r}} u \le C \left[ (R-r)^{-\frac{n}{p}} \|u\|_{L^{p}(B_{R})} + R^{\delta} \|\rho_{*}(x,\vec{F})\|^{\frac{1}{\gamma-1}}_{L^{q}(B_{R})} + R^{\gamma^{\prime} \delta} \|f\|_{L^{\frac{q}{\gamma^{\prime}}}}^{\frac{1}{\gamma-1}} \right]
\end{equation}
where $\delta = 1 - \frac{n}{q(\gamma-1)} > 0$.
\end{theorem}

The second main result is an inf-bound for non-negative weak supersolutions of \eqref{e:inhomo}.

\begin{theorem}
Suppose $1 < \gamma < n$ and $u \in W^{1,\gamma}(\Omega)$ is a non-negative supersolution to \eqref{e:inhomo}, for some $\rho$ satisfying \eqref{e:meas} - \eqref{e:eb}. Assume $\vec{F}$, $f$, and $q$ are as in \eqref{e:Ff}. If $0 < r < R < 1$ and $B_{R} \subset \Omega$, then for all $0 < p < \frac{n(\gamma-1)}{n-\gamma}$ and all $0 < \theta < \tau < 1$, there exists $C = C(n, \gamma, \nu, \Lambda, q, p, \theta, \tau) > 0$ so that

$$
\inf_{B_{\theta R}} u + R^{\delta} \| \rho_{*}(\vec{F})\|_{L^{q}(B_{R})} + R^{\gamma^{\prime} \delta} \| f\|_{L^{\frac{q}{\gamma^{\prime}}}(B_{R})} \ge C R^{-\frac{n}{p}} \|u\|_{L^{p}(B_{\tau R})},
$$
where $\delta = 1 - \frac{n}{q(\gamma-1)}$.
\end{theorem}

In Section \ref{s:theorems} we show that these two main theorems can be combined to achieve many more concrete tools like Moser's weak Harnack inequality (Theorem \ref{t:iwh}). As further consequences when $\vec{F},f \equiv 0$, the usual collection of tools from elliptic theory hold: a strong maximum principle (Theorem \ref{t:mp}), improvement of oscillation and $C^{\alpha}$-regularity of solutions (Theorem \ref{t:ioo} and Corollary \ref{c:calpha}), and a Liouville-type theorem, Theorem \ref{t:liouville}.

Of independent interest is the following Bernstein-type theorem which follows from Theorem \ref{t:liouville} and Example \ref{x:pde2}(ii):

\begin{theorem}
Suppose $\Sigma \defeq \{ (x, u(x)) \in \R^{n+1} : x \in \R^{n} \}$ is stationary with respect to the energy \eqref{e:lpsenergy}. If $u$ is Lipschitz and bounded above or below, then $\Sigma= \R^{n} \times \{u(0)\}$.
\end{theorem}

We note that the requirements that $\rho \in C^{1}$ and $\rho$ be strictly strictly convex are minimal qualitative hypotheses necessary to have any hope of regularity for the Finsler $\gamma$-harmonic functions. Indeed, \eqref{e:thesimplepde} is not well-defined if $\rho$ is not differentiable. On the other hand its well-known that strict convexity is a necessary condition for regularity, see e.g., \cite[Remark 20.4]{maggi2012sets}. In terms of additional hypotheses, considering \eqref{e:inhomo}, the condition that $\{ \rho_{x}\}$ are equi-nondegenerate and equi-bounded is a way of fixing the ``homogeneity'' of the PDE. It is known from mixed-growth problems that if the homogeneity of the PDE vary too much, one loses any hope of a robust regularity theory, see for instance \cite{marcellini1991regularity,mingione2006regularity}. 

To the best of the author's knowledge, the existing literature on the Finsler $\gamma$-Laplacian, is focused on spectrum of the eigenvalues see e.g., \cite{kristaly2022nonlinear} and references therein. There is a small amount of literature dedicated to the regularity of the Finsler Laplacian \cite{ferone2009remarks,li2021extremal,fazly2022partial}. However, due to the non-linearity of the Finsler Laplacian\footnote{So long as $\rho(\xi) \neq \langle A\xi, \xi \rangle^{1/2}$ for some positive definite matrix $A$} without assuming strong non-degeneracy conditions on $\rho_{x}$ (see \cite[Page 180]{fazly2022partial} and the even stronger \cite[(1.2)]{fazly2022partial}) in stark contrast to the $p=2$ case of the $p$-Laplacian, the case where $\gamma = 2$ is not particularly special for the Finsler $\gamma$-Laplacian.

\begin{remark}
We conclude our discussion of the literature about regularity for the Finsler Laplacian by making a few observations about the condition used to prove the mean-value property in \cite[Section 5]{ferone2009remarks}. Therein it is assumed that 
\begin{equation} \label{e:limited}
\langle D\rho(\xi), D\rho_{*}(\zeta) \rangle = \frac{ \langle \xi, \zeta \rangle }{\rho(\xi) \rho_{*}(\zeta)} \qquad \forall \xi,\zeta \in \R^{n} \setminus \{0\},
\end{equation}
where $\rho_{*}$ is the dual norm to $\rho$, see \eqref{e:dualnorm}. This condition is also used throughout \cite{fazly2022partial} and \cite{li2021extremal}, and it has been partially addressed in \cite{cozzi2016monotonicity}. In \cite[Theorem 1.2]{cozzi2016monotonicity} it is shown that the only norms satisfying the condition \eqref{e:limited} are precisely those of the form $\displaystyle \rho(\xi) = \langle A \xi, \xi \rangle^{1/2}$ for some strictly positive definite matrix $M$. In light of this, \cite{cozzi2016monotonicity} uses the weaker formulation \eqref{e:dualmono} and in \cite[Section 7]{cozzi2016monotonicity} they show that in the plane there are explicit conditions one can use to create more norms that satisfy the weaker condition 
\begin{equation} \label{e:dualmono}
\langle D\rho(\xi), D\rho_{*}(\zeta) \rangle \langle \xi,\zeta \rangle \ge 0 \qquad \forall \xi,\zeta \in \R^{n} \setminus \{0\}.
\end{equation}

However, a consequence of the $k=n-1 \ge 2$ case of \cite[Theorem page 437]{allard1974characterization} goes even further to say that if $g$ is any $1$-homogeneous $C^{1}$ function on $\R^{n} \setminus \{0\}$ with $n \ge 3$, then there exists $\xi,\zeta \in \R^{n} \setminus \{0\}$ so that
 \begin{equation*} 
\langle D\rho(\xi), Dg(\zeta) \rangle \langle \xi, \zeta \rangle < 0,
\end{equation*}
unless $\rho(\xi) = |L \xi|$ and $g(\xi) = |L^{t} \xi|$ for some matrix $L$. In particular, when $n \ge 3$ and $\gamma = 2$, the assumption \eqref{e:limited} or \eqref{e:dualmono} implies the Finsler $\gamma$-Laplacian is as in Example \ref{x:pde1}(iv) and simplifies to a constant-coefficient, linear, divergence form PDE.
\end{remark}

Finally, in the preparation of this paper the author learned about the recent work in \cite{benyaiche2020weak}. After discovering that paper we learned that there is also a wealth of literature about minimizers with Orlicz-type growth conditions. See, for instance, \cite{harjulehto2017holder,benyaiche2019harnack,arriagada2018harnack}, and the citations therein. In the special case where $\vec{F},f = 0$ the works of \cite{benyaiche2019harnack,benyaiche2020weak} recover some of the results of this paper. Using Orlicz spaces they are able to replace the assumptions on fixed homogeneity with some notion of upper and lower (almost)-homogeneity. In the isotropic setting, \cite{arriagada2018harnack} shows a Harnack inequality similar to the one in Theorem \ref{t:iwh}. Meanwhile \cite{toivanen2012harnack} proves a general Harnack inequality, while assuming that the terms $\vec{F}, f$ are in $L^{\infty}$.

Also during preparation of this paper \cite{fazly2020partial} proved a Liouville-type theorem when $\gamma = 2$ and $\rho$ has strictly positive Hessian in the sense that 

$$
\nu^{2} |\zeta|^{2} \le (\partial_{\xi_{i}} \partial_{\xi_{j}} \rho)(\xi)\zeta_{i} \zeta_{j} \le \Lambda |\zeta|^{2} \qquad \forall \zeta \in \xi^{\perp}.
$$

\subsection*{Acknowledgments} The author would like to thank Silvia Ghinassi and Tatiana Toro for many useful discussions in preparation of these works, and also Louis Dupaigne for helpful comments on the first draft of this paper.

\section{Notation and Preliminaries}

Throughout we will suppose $\rho: \Omega \times \rn \to [0, \infty)$ is so that 
\begin{equation} \label{e:meas}
\begin{cases}
\rho(x, \cdot ) \in C^{1}(\rn \setminus \{0\}) & \forall x \in \Omega \\ 
\rho(\cdot, \xi) \in L^{\infty}(\Omega)  & \forall \xi \in \rn \setminus \{0\}.
\end{cases}
\end{equation}
and that $\rho$ positively $1$-homogeneous function in its second variable, i.e.,
\begin{equation} \label{e:p1h}
\rho( x, \lambda \xi) = \lambda \rho(x, \xi) \qquad \forall \lambda > 0, \quad  \forall x \in \Omega, \quad \forall \xi \in \rn \setminus \{0\}.
\end{equation}
We further assume that $\rho(x, \cdot)$ is strictly convex in the sense that
\begin{equation} \label{e:cx}
\rho(x, \xi_{1} + \xi_{2}) < \rho(x, \xi_{1}) + \rho(x, \xi_{2}) \quad \forall x \in \Omega, \quad \forall \xi_{1} \neq \lambda \xi_{2} \in \rn \setminus \{0\} .
\end{equation}
Finally, we assume there exists $0 < \nu \le \Lambda < \infty$ independent of $x$ so that
\begin{equation} \label{e:eb}
 \nu \le \rho(x, \xi) \le \Lambda \qquad \forall |\xi| = 1, \qquad \forall x \in \Omega.
\end{equation}

We will let $\rho_{x}$ denote $\rho(x, \cdot)$ to simplify notation. Namely, $(D \rho_{x})(Du) = (D \rho(x, \cdot))(Du)$. We say that $u \in W^{1,\gamma}(\Omega)$ is a subsolution (supersolution) to \eqref{e:inhomo} if

\begin{equation*} 
\int_{\Omega} \Langle \rho_{x}(Du)^{\gamma-1} (D \rho_{x})(Du), D \varphi \Rangle \le (\ge) \int_{ \Omega} \Langle \vec{F}, D \varphi \Rangle + f \varphi 
\end{equation*}
for all non-negative $\varphi \in W^{1,\gamma^{\prime}}_{0}(\Omega)$. We say that $u$ is a solution if it is both a subsolution and supersolution.

Given a real number, say $\alpha$, in $(1, \infty)$ or $[1,n)$, we will respectively always let $\alpha^{\prime}$ and $\alpha^{*}$ denote the Holder and Sobolev exponents. That is,
$$
\frac{1}{\alpha} + \frac{1}{\alpha^{\prime}} = 1 \qquad \alpha^{*} = \frac{n \alpha}{n-\alpha}.
$$

\begin{theorem}[Gagliado-Nirenberg-Sobolev]
If $u \in W^{1,\gamma}_{0}(\Omega)$ then there exists $C = C(n,\gamma) > 0$ so that
$$
\| u\|_{L^{\gamma^{*}(\Omega)}} \le C \|Du\|_{L^{\gamma}(\Omega)} 
$$
\end{theorem}

\begin{theorem}[Poincare in a ball]
For each $1 \le \gamma < n$ there exists a $C = C(n,\gamma)$ so that
$$
\left( r^{-n} \int_{B(x,r)} |f-(f)_{x,r}|^{\gamma^{*}} dy \right)^{\frac{1}{\gamma^{*}}} \le C_{2} r^{1- \frac{n}{\gamma}} \left( \int_{B(x,r)} |Df|^{\gamma} dy \right)^{\frac{1}{\gamma}}
$$
\end{theorem}

\begin{remark} \label{r:sobolev}
Since $\rho$ will always satisfies \eqref{e:meas}, \eqref{e:p1h}, and \eqref{e:cx}, if it also solves the first inequality in \eqref{e:eb} then the Sobolev embedding theorem can be re-written as 
$$
\|u\|_{L^{\gamma^{*}}(\Omega)} \le C \nu^{-1} \| \rho_{x}(Du)\|_{L^{\gamma}(\Omega)}.
$$
Similarly Poincare in a ball can be re-written as
$$
\left( r^{-n} \int |f-(f)_{x,r}|^{\gamma^{*}} dy \right)^{\frac{1}{\gamma^{*}}} \le C_{2} \nu^{-1} r^{1- \frac{n}{\gamma}} \left( \int_{B(x,r)} \rho_{x}(Df)^{\gamma} dy \right)^{\frac{1}{\gamma}}.
$$
\end{remark}

Given $\rho : \Omega \times \rn \setminus \{0\} \to (0,\infty)$ as in \eqref{e:meas}, \eqref{e:p1h}, \eqref{e:cx}, define $\rho_{*} : \Omega \times \rn \setminus \{0\} \to (0, \infty)$ by
\begin{equation} \label{e:dualnorm}
\rho_{*}(x, \xi^{*}) = \sup_{\rho(x,\xi) < 1} \xi \cdot \xi^{*}.
\end{equation}
That is, $\rho_{*}(x, \cdot)$ is the convex dual of $\rho(x, \cdot)$ for all $x \in \Omega$. We record for the reader's convenience a few facts that are frequently used:
\begin{proposition} \label{p:dualprops}
Let $a,b,c, \epsilon > 0$, $\alpha \in (1, \infty)$, $\rho$ as in \eqref{e:meas}, \eqref{e:p1h}, and \eqref{e:cx}. Suppose $\xi_{1}, \xi_{2} \in \rn \setminus \{0\}$. 
\begin{itemize}
\item Young's inequality says
$$
abc \le a \epsilon^{\alpha} \frac{b^{\alpha}}{\alpha} + a \epsilon^{-\alpha^{\prime}} \frac{c^{\alpha^{\prime}}}{\alpha^{\prime}}.
$$
\item Fenchel's inequality guarantees
$$
\xi_{1} \cdot \xi_{2} \le \rho(\xi_{1}) \rho_{*}(\xi_{2})
$$
\item It holds,
\begin{equation} \label{e:eq1}
\rho_{*}(x,D \rho_{x}(\xi_{1})) \equiv 1.
\end{equation}
\item The dual norm $\rho_{*}$ also satisfies \eqref{e:meas}, \eqref{e:p1h}, \eqref{e:cx}. Moreover,
$$
(D \rho) \circ (D \rho_{*}) (\xi^{*}) = \frac{\xi^{*}}{\rho_{*}(\xi^{*})} \quad \text{and} \quad (D \rho_{*}) \circ (D \rho)(\xi) = \frac{\xi}{\rho(\xi)} \qquad \forall \xi, \xi^{*} \in \rn \setminus \{0\}.
$$

\item If $\rho$ satisfies \eqref{e:eb} then for all $|\xi| = 1$,
$$
\Lambda^{-1} \le \rho_{*}(\xi) \le \nu^{-1}.
$$
In particular, $\vec{F} \in L^{q}(\Omega)$ if and only if $\rho_{*}(\cdot,\vec{F}) \in L^{q}(\Omega)$
\end{itemize}
\end{proposition}

We prove the following Cacciopolli type inequality to show the usefulness of the dual norm and corresponding estimates in Proposition \ref{p:dualprops}.

\begin{lemma}
If $u$ is a subsolution to \eqref{e:inhomo} with $\vec{F},f \equiv 0$ and $1 < \gamma < \infty$, then
$$
\| \eta \rho_{x}(D u) \|_{L^{\gamma}(\Omega)} \le C(n,\gamma) \| u \rho_{x}(D \eta)\|_{L^{\gamma}(\Omega)}.
$$
\end{lemma}

\begin{proof}
Consider $\varphi = \eta^{\gamma} u$. Then $D \varphi = \gamma \eta^{\gamma-1} u D \eta+ \eta^{\gamma} Du$. So, using the $1$-homogeneity of $\rho$ and Fenchel's inequality, 
\begin{align*}
& \Langle \rho_{x}(Du)^{\gamma-1} (D \rho_{x})(Du) , D \varphi \Rangle \\ & \ge- \gamma (\eta \rho_{x}(Du))^{\gamma-1} \rho_{*}(x,(D \rho_{x})(Du)) u \rho_{x}(D \eta) + \eta^{\gamma} \rho_{x}(Du)^{\gamma}.
\end{align*}
Using $u$ is a subsolution with $\vec{F},f \equiv 0$, and \eqref{e:eq1} yields
\begin{align*}
\int_{\Omega} \eta^{\gamma} \rho_{x}(Du)^{\gamma} & \le \gamma \int_{\Omega} ( \eta \rho_{x}(Du))^{\gamma-1} u \rho_{x}(D \eta) \\
& \le \left( \int ( \eta \rho(Du))^{\gamma} \right)^{1 - \frac{1}{\gamma}} \left( \int_{\Omega} u^{\gamma} \rho(D \eta)^{\gamma} \right)^{\frac{1}{\gamma}}.
\end{align*}
Dividing completes the proof.
\end{proof}

Another simple consequence of the estimates in Proposition \ref{p:dualprops}, which should be compared with \cite[Section 2]{ferone2009remarks} is

\begin{remark}[Fundamental Solutions] \label{r:fs}
If $\rho$ is a $C^{1}$ and strictly convex norm, then up to a multiplicative constant,
$$
\begin{cases}
\rho_{*}(x)^{\frac{\gamma-n}{\gamma-1}} & n \neq \gamma > 1 \\
\ln ( \rho_{*}(x)) & n = \gamma >1,
\end{cases}
$$
is the fundamental solution to the Finsler $\gamma$-Laplacian.
\end{remark}

We quickly verify Remark \ref{r:fs} in the case $n \neq \gamma$: 

If $1 < \gamma \neq n$, $F(x) = \frac{\gamma-1}{\gamma-n}\rho_{*}(x)^{\frac{\gamma-n}{\gamma-1}}$ then $\divr \left( \rho(DF)^{\gamma-1}(D \rho)(DF) \right) = 0$. This generalizes the so-called fundamental solutions for the $p$-Laplacian.
Indeed, $DF(x) = \rho_{*}(x)^{\frac{1-n}{\gamma-1}} (D \rho_{*}(x))$ so
$$
\begin{cases}
\rho(DF(x))^{\gamma-1} = \rho_{*}(x)^{1-n} \\
(D \rho)(DF(x)) = \frac{x}{\rho_{*}(x)}.
\end{cases}
$$
Hence, $\rho(DF)^{\gamma-1}(D \rho)(Df) = x \rho_{*}(x)^{-n}$ which has divergence zero in $\R^{n}$ for any $1$-homogeneous, $C^{1}$ function $\rho_{*}$.

Finally, we recall a technical lemma for later use, see e.g., \cite[Lemma 4.19]{han2011elliptic}
\begin{lemma} \label{l:ibs}
Let $\omega, \sigma$ be non-decreasing functions in $(0,R]$. Suppose there exists $0 < \tau, \tilde \delta < 1$ so that for all $r \le R$,
$$
\omega( \tau r) \le \tilde \delta \omega(r) + \sigma(r).
$$
Then for any $\mu \in (0,1)$ and $r \le R$
$$
\omega(r) \le C \left\{ \left( \frac{r}{R} \right)^{\alpha} \omega(R) + \sigma(r^{\mu} R^{1-\mu}) \right\}
$$
where $C = C(\tilde \delta, \tau)$ and $\alpha = \alpha( \tilde \delta, \tau, \mu)$.
\end{lemma}

\section{Main results} \label{s:theorems}

In this section we focus on functions $u$ that solve
\begin{equation} \label{e:dt1}
\int_{\Omega} \Langle \rho_{x}(Du)^{\gamma-1} (D \rho_{x})(Du) , D \varphi \Rangle dx = \int_{\Omega} \Langle \vec{F}, D \varphi \Rangle + f \varphi  \qquad \forall \varphi \in W^{1, \gamma^{\prime}}_{0}(\Omega).
\end{equation}
As a corollary of these results, we can answer further questions about functions $u$ that solve 
\begin{equation} \label{e:homo}
\int_{\Omega} \Langle \rho_{x}(Du)^{\gamma-1} (D \rho_{x})(Du) , D \varphi \Rangle dx = 0.
\end{equation}

We begin with a Caccioppoli inequality when $\vec{F},f \not\equiv 0$.

\begin{theorem} \label{t:crhs}
Suppose $u \in W^{1, \gamma}(\Omega)$ is a subsolution of \eqref{e:dt1} with $1 < \gamma < \infty$ and $\rho : \Omega \times \rn \setminus \{0\}$ satisfies \eqref{e:meas}, \eqref{e:p1h}, \eqref{e:cx}, and \eqref{e:eb}. Assume $\vec{F},f \in L^{\tilde \gamma}(\Omega)$ where $\tilde \gamma = \max \{\gamma, \gamma^{\prime} \}$. If $B_{2R} \subset \Omega$ and $0 < R \le 10$ then,
\begin{equation} \label{e:crhs}
\| \rho_{x}(Du)\|_{L^{\gamma}(B_{R})} \le c_{\gamma,\Lambda} \left[ R^{-1} \|u\|_{L^{\gamma}}(B_{2R}) + \| \rho_{*}(x,\vec{F}) \|_{L^{\gamma^{\prime}}(B_{2R})}^{\frac{1}{\gamma-1}} + R^{\frac{1}{\gamma-1}} \| f \|_{L^{\gamma^{\prime}}(B_{2R})}^{\frac{1}{\gamma-1}} \right].
\end{equation}
\end{theorem}

\begin{remark}
Note, it is necessary for $\vec{F}, f \in L^{\gamma}$ (resp., $\vec{F}, f \in L^{\gamma^{\prime}}$) for the equation (resp., conclusion) to make sense\footnote{Technically, by Sobolev embedding we only need $f \in L^{(\frac{\gamma^{\prime} n}{n-\gamma^{\prime}} )^{\prime}} \cap L^{\gamma^{\prime}}$.}. 
\end{remark}

\begin{proof}
Suppose without loss of generality, $R = 1$. Consider the test function $\varphi = \eta^{\gamma} u$ for some $\eta \in C^{1}_{0}(B_{2})$ to be chosen later. Note, 
$$
D \varphi = \eta^{\gamma} Du + \gamma u \eta^{\gamma-1} D \eta.
$$ 
Choose $\epsilon > 0$ so that $(\gamma-1) \epsilon^{\gamma^{\prime}} = \frac{1}{2} \rho( \eta Du)^{\gamma}$. Using the $1$-homogeneity of $\rho_{x}$, Fenchel's inequality, Young's inequality, and \eqref{e:eq1} compute
\begin{align} 
\nonumber \langle \rho_{x}(Du)^{\gamma-1} (D\rho_{x})(Du), D \varphi \rangle & \ge \rho_{x}( \eta Du)^{\gamma} -  \gamma \rho_{x}(\eta Du)^{\gamma-1} \rho_{*}(x,(D \rho)(Du)) \rho_{x}(u D \eta) \\
\nonumber & \ge \rho_{x}( \eta Du)^{\gamma} - \gamma \left[ \frac{ \epsilon^{\gamma^{\prime}} \rho_{x}( \eta D u)^{\gamma}}{\gamma^{\prime}} + \frac{ \rho_{x}(u D \eta)^{\gamma}}{\epsilon^{\gamma} \gamma} \right] \\
\label{e:c1} & \ge \frac{\rho_{x}( \eta Du)^{\gamma}}{2} - c_{\gamma} \rho_{x}(u D \eta)^{\gamma}.
\end{align}
On the other hand, when choosing $\epsilon> 0$ so that $\gamma^{-1} \epsilon^{\gamma} = 1/4$, Fenchel's and Young's inequalities ensure
\begin{align}
\nonumber \langle \vec{F}, D \varphi \rangle & \le \gamma \left(\rho_{*}(x,\vec{F})\eta^{\gamma-1}\right) \rho_{x}(u D \eta) + \eta^{\gamma} \left( \rho_{*}(\vec{F})  \rho(Du) \right) \\
\nonumber & \le \gamma \left[ \frac{\rho_{*}(x,\vec{F})^{\gamma^{\prime}} \eta^{\gamma}}{ \gamma^{\prime}} +  \frac{\rho_{x}(u D \eta)^{\gamma}}{\gamma} \right] + \eta^{\gamma} \left[ \frac{\rho_{*}( x,\vec{F})^{\gamma^{\prime}}}{\epsilon^{\gamma^{\prime}} \gamma^{\prime}} + \epsilon^{\gamma} \frac{ \rho_{x}(Du)^{\gamma}}{\gamma} \right] \\
& \label{e:c2}= \frac{\rho_{x}( \eta Du)^{\gamma}}{4} + c_{\gamma} \rho_{*}(x,\vec{F})^{\gamma^{\prime}} \eta^{\gamma} + \rho_{x}(u D \eta)^{\gamma}.
\end{align}
Combining \eqref{e:c1}, \eqref{e:c2}, and \eqref{e:dt1} yields,
\begin{align} 
\nonumber \int \rho_{x}(\eta Du)^{\gamma} & \le c_{\gamma} \left[ \int \rho_{*}(x,\vec{F})^{\gamma^{\prime}} \eta^{\gamma} + \int \rho_{x}(u D \eta)^{\gamma} + \int f \eta^{\gamma} u \right] \\
\label{e:c4} & \le c_{\gamma} \left[ \int \rho_{*}(x,\vec{F})^{\gamma^{\prime}} \eta^{\gamma} + \int u^{\gamma}(\eta^{\gamma} + \rho_{x}( D \eta)^{\gamma}) +  \int  \eta^{\gamma } f^{\gamma^{\prime}}  \right].
\end{align}
Choosing $0 \le \eta \le 1$, $\eta \equiv 1$ on $B_{1}$, $\eta \equiv 0$ on $B_{2}^{c}$ and $|D \eta| \le 2$ we find
\begin{equation*}
\| \rho_{x}( Du)\|_{L^{\gamma}(B_{1})} \le c_{\gamma, \Lambda,} \left[ \|u\|_{L^{\gamma}(B_{2})} + \| \rho_{*}(x,\vec{F}) \|_{L^{\gamma^{\prime}}(B_{2})}^{\gamma^{\prime}-1} + \| f \|_{L^{\gamma^{\prime}}(B_{2})}^{\gamma^{\prime}-1} \right].
\end{equation*}
Equation \eqref{e:crhs} is recovered by scaling.
\end{proof}

We note that in Theorem \ref{t:crhs}, the fact that $\rho$ can depend on $x$ never needs to be dealt with separately. This is unsurprising because conditions \eqref{e:p1h}, \eqref{e:cx}, and Fenchel's inequality are used at a pointwise level while \eqref{e:eb} is used at a global level, see Remark \ref{r:sobolev}. Hence, to simplify notation, we only explicitly write-out the dependence of $\rho$ on $x$ in the statements of theorems and suppress this dependence throughout all remaining proofs.

\begin{theorem} \label{t:sup}
Suppose $\rho : \Omega \times \rn \setminus \{0\}$ satisfies \eqref{e:meas}, \eqref{e:p1h}, \eqref{e:cx},  and \eqref{e:eb}. Let $u$ be a subsolution to \eqref{e:dt1} and fix $1 < \gamma < n$. If $\vec{F},f$ and $q$ are as in \eqref{e:Ff}, $0 < r < R < 1$, and $\overline{B_{R}} \subset \Omega$ then there exists some $C = C(n,\nu, \Lambda,\gamma,q,p)$ and $\delta = 1 - \frac{n}{q(\gamma-1)} > 0$ so that 
\begin{align*}
\sup_{B_{r}} u^{+} & \le C  \bigg[ \frac{\| u^{+}\|_{L^{p}(B_{R})}}{(R-r)^{\frac{n}{p}} } + R^{\delta} \|\rho_{*}(x,\vec{F})\|_{L^{q}(B_{R})}^{\frac{1}{\gamma-1}} +  R^{\gamma^{\prime} \delta}\|f\|_{L^{\frac{q}{\gamma^{\prime}}}(B_{R})}^{\frac{1}{\gamma-1}} \bigg]
\end{align*}
\end{theorem}

\begin{proof}
We consider the test function $\varphi = \eta^{\gamma} v^{\beta} \bar{u}$ for $\beta \ge 0$ where $\bar{u} = u^{+} + \bar{k}$ and $v = \min\{u,m\}$ for some $0 < \bar{k} < m < \infty$, $\bar{k}$ to be chosen later. Notice 
$$
D \varphi = \gamma \eta^{\gamma-1} v^{\beta} \bar{u} D \eta + \beta \eta^{\gamma} v^{\beta} \bar{U} Dv + \eta^{\gamma} v^{\beta} D \bar{u}.$$

We wish to expand \eqref{e:dt1} with this choice of $\varphi$. To this end, first observe $1$-homogeneity, i.e., \eqref{e:p1h} ensures
\begin{align} \label{e:dt2}
\nonumber \big \langle \rho(Du)^{\gamma-1} (D \rho)(Du), & \beta v^{\beta-1} \bar{u} \eta^{\gamma} D v + \eta^{\gamma} v^{\beta} D \bar{u} \big \rangle \\
& = \beta \rho(Dv)^{\gamma} v^{\beta} \eta^{\gamma} + \rho(D\bar{u})^{\gamma} \eta^{\gamma} v^{\beta}.
\end{align}
Next we apply Fenchel's and Young's inequalities in combination with \eqref{e:eq1} for some $\epsilon = \epsilon(\gamma) > 0$ so that $ (\gamma-1) \epsilon^{\gamma^{\prime}} = 1/2$. Then,
\begin{align} 
\nonumber \big \langle \rho(Du)^{\gamma-1} (D \rho)(Du),&  \gamma \eta^{\gamma-1} v^{\beta} \bar{u} D \eta \big \rangle  \\
\nonumber & \ge - \gamma v^{\beta} \left( \rho(D \bar{u}) \eta) \right)^{\gamma-1} \rho_{*}(D\rho(Du)) \left( \rho(D \eta) \bar{u} \right)      \\
\label{e:dt3}  & \ge - \gamma v^{\beta} \left[  \epsilon^{\gamma^{\prime}} \frac{\rho(D \bar{u})^{\gamma} \eta^{\gamma}}{\gamma^{\prime}} + \frac{ \rho(D \eta)^{\gamma} \bar{u}^{\gamma}}{ \epsilon ^{\gamma} \gamma} \right].
\end{align}
Since $\frac{\gamma}{\gamma^{\prime}} = \gamma-1$, this choice of $\epsilon$ ensures \eqref{e:dt3} becomes
\begin{align} \label{e:dt4}
 \Langle \rho(Du)^{\gamma-1} (D \rho)(Du), \gamma \eta^{\gamma-1} v^{\beta} \bar{u} D \eta \Rangle & \ge  -  \frac{v^{\beta} \rho(D \bar{u})^{\gamma} \eta^{\gamma}}{2} - c_{\gamma} v^{\beta} \rho(D\eta)^{\gamma} \bar{u}^{\gamma}, 
\end{align}
where $c_{\gamma}$ may change depending on the line, but depends only on $\gamma$.

Now we look at the righthand side. We split this into two pieces and treat the first piece in much the same fashion as above. 
\begin{align}
\nonumber \big\langle \vec{F}, & \beta v^{\beta-1} \bar{u} \eta^{\gamma} Dv + \eta^{\gamma} v^{\beta} D \bar{u} \big\rangle  \le  \rho_{*}(\vec{F}) \left[ \beta v^{\beta} \eta^{\gamma} \rho(Dv) \right] + \rho_{*}(\vec{F}) \left[ \eta^{\gamma} v^{\beta} \rho(D \bar{u}) \right] \\
\nonumber & = (\eta^{\gamma} v^{\beta}) \left[ \left(\beta\right) \left( \rho_{*}(\vec{F}) \right) \left(  \rho(Dv) \right) + \rho_{*}(\vec{F}) \rho(D \bar{u}) \right] \\
\nonumber & \le   \eta^{\gamma} v^{\beta} \left[ \left( \beta \frac{\rho_{*}(\vec{F})^{\gamma^{\prime}}}{\epsilon_{1}^{\gamma^{\prime}} \gamma^{\prime}} + \beta \epsilon_{1}^{\gamma} \frac{ \rho(D v)^{\gamma}}{\gamma} \right) + \left( \frac{ \rho_{*}(\vec{F})^{\gamma^{\prime}}}{ \epsilon_{2}^{\gamma^{\prime}} \gamma^{\prime}} + \epsilon_{2}^{\gamma} \frac{ \rho(D \bar{u})^{\gamma}}{\gamma} \right) \right] \\
& \label{e:dt5} =  \eta^{\gamma} v^{\beta} \rho_{*}(\vec{F})^{\gamma^{\prime}}\left( \frac{\beta}{\epsilon_{1}^{\gamma^{\prime}} \gamma^{\prime}} + \frac{1}{\epsilon_{2}^{\gamma^{\prime}} \gamma^{\prime}} \right) + \frac{\beta \epsilon_{1}^{\gamma} }{\gamma} \eta^{\gamma} v^{\beta} \rho(Dv)^{\gamma} + \frac{\epsilon_{2}^{\gamma}}{2} \eta^{\gamma} v^{\beta} \rho(D \bar{u})^{\gamma}.
\end{align}
Since we want to absorb the last two terms of \eqref{e:dt5} into \eqref{e:dt2} we choose $\epsilon_1, \epsilon_{2}$ so that $\gamma^{-1} \beta \epsilon_{1}^{\gamma} = \frac{\beta}{2}$ and $\gamma^{-1} \epsilon_{2}^{\gamma} = \frac{1}{4}$. The need for choosing $1/4$ for the $\epsilon_{2}$ coefficient is due to the fact that we'll also be absorbing the $\rho(D \bar{u})$-term from \eqref{e:dt4} into \eqref{e:dt2}. We note both $\epsilon_{1}$ and $\epsilon_{2}$ depend solely on $\gamma$. All-in-all this allows us to re-write \eqref{e:dt5} as
\begin{align} \label{e:dt6}
 \Langle  \vec{F}, \beta v^{\beta-1} \bar{u} \eta^{\gamma} Dv + \eta^{\gamma} v^{\beta} D \bar{u} \Rangle & \le  c_{\gamma}(1 + \beta) \eta^{\gamma} v^{\beta} \rho_{*}(\vec{F})^{\gamma^{\prime}} + \frac{\beta}{2} \eta^{\gamma} v^{\beta} \rho(Dv)^{\gamma} \\ 
 \nonumber & + \frac{1}{4} \eta^{\gamma} v^{\beta} \rho(D \bar{u})^{\gamma}.
\end{align}

We now deal with the final term via Fenchel and Young's inequalities
\begin{align}
\nonumber \Langle \vec{F}, \gamma \eta^{\gamma-1} v^{\beta} \bar{u} D \eta \Rangle & \le  \gamma \rho_{*}(\vec{F}) \eta^{\gamma-1} v^{\beta} \bar{u} \rho(D \eta) \\
\nonumber & \le  \gamma v^{\beta} \left[ \frac{ \rho_{*}(\vec{F})^{\gamma^{\prime}} \eta^{\gamma} }{\gamma^{\prime}} + \frac{ \bar{u}^{\gamma} \rho(D \eta)^{\gamma}}{\gamma} \right] \\
\label{e:dt7} & =  (\gamma-1) \rho_{*}(\vec{F})^{\gamma^{\prime}} \eta^{\gamma} v^{\beta} + \bar{u}^{\gamma} \rho(D \eta)^{\gamma} v^{\beta}.
\end{align}

Finally, plugging \eqref{e:dt2}, \eqref{e:dt4}, \eqref{e:dt6}, and \eqref{e:dt7} into \eqref{e:dt1} yields  
\begin{align} 
\nonumber \frac{\beta}{2}& \int_{\Omega} \rho(Dv)^{\gamma} v^{\beta} \eta^{\gamma} + \frac{1}{4} \int_{\Omega} \rho(D \bar{u})^{\gamma} \eta^{\gamma} v^{\beta} \\
\nonumber &  \le  c_{\gamma} \left[ \int_{\Omega} v^{\beta} \rho(D \eta)^{\gamma} \bar{u}^{\gamma} + \int_{\Omega} \eta^{\gamma} v^{\beta} \rho_{*}(\vec{F})^{\gamma^{\prime}} + \int_{\Omega} f \eta^{\gamma} v^{\beta} \bar{u} \right] \\
 \label{e:dt8} & \le c_{\gamma} \left[ \int_{\Omega} v^{\beta} \rho(D \eta)^{\gamma} \bar{u}^{\gamma} + (1+\beta)\int_{\Omega} \eta^{\gamma} v^{\beta} \frac{\bar{u}^{\gamma}}{\bar{k}^{\gamma}} \rho_{*}(\vec{F})^{\gamma^{\prime}} + \int_{\Omega}  \frac{f}{\bar{k}^{\gamma-1}} \eta^{\gamma} v^{\beta} \bar{u}^{\gamma} \right].
\end{align}

The last inequality used $\bar{u} \ge \bar{k}$. Now set $w= v^{\beta/\gamma} \bar{u}$, and observe
$$
Dw = \frac{\beta}{\gamma} v^{\frac{\beta}{\gamma}-1} \bar{u} Dv + v^{\frac{\beta}{\gamma}} D \bar{u} = \frac{\beta}{\gamma} v^{ \frac{\beta}{\gamma}} Dv + v^{\frac{\beta}{\gamma}} D \bar{u}.
$$
Since $\rho(\xi_{1} + \xi_{2}) \le \rho(\xi_{1}) + \rho(\xi_{2})$ for all $\xi_{1}, \xi_{2}$ it follows that
\begin{align*}
\eta^{\gamma} \rho(D w)^{\gamma} &\le \eta^{\gamma} \left( \frac{\beta}{\gamma} v^{\frac{\beta}{\gamma}} \rho(Dv) + v^{\frac{\beta}{\gamma}} \rho(D \bar{u}) \right)^{\gamma} \\
& \le 2^{\gamma-1} \left( \left( \frac{\beta}{\gamma} \right)^{\gamma} \eta^{\gamma} v^{\beta} \rho(Dv)^{\gamma} + \eta^{\gamma} v^{\beta} \rho(D \bar{u})^{\gamma} \right).
\end{align*}
In particular, this guarantees that for some $c_{\gamma}$
\begin{equation} \label{e:dt9}
\int_{\Omega} \eta^{\gamma} \rho(Dw)^{\gamma} \le c_{\gamma} (1 + \beta^{\gamma-1}) \left[ \frac{\beta}{2} \int_{\Omega} \rho(Dv)^{\gamma} v^{\beta} \eta^{\gamma} + \frac{1}{4} \int_{\Omega} \rho(D \bar{u})^{\gamma} \eta^{\gamma} v^{\beta} \right].
\end{equation}
Combining \eqref{e:dt8} and \eqref{e:dt9} yields
\begin{equation} \label{e:dt10} 
\int \eta^{\gamma} \rho(Dw)^{\gamma} \le c_{\gamma}(1 + \beta^{\gamma}) \left[ \int_{\Omega} w^{\gamma} \rho(D \eta)^{\gamma}  +\int_{\Omega} (\eta w)^{\gamma} \left( \frac{\rho_{*}(\vec{F})^{\gamma^{\prime}}}{\bar{k}^{\gamma}} + \frac{f}{\bar{k}^{\gamma-1}} \right) \right].
\end{equation}
Due to the observation that $\rho(D(\eta w))^{\gamma} \le 2^{\gamma-1} (w^{\gamma} \rho(D \eta)^{\gamma} + \eta^{\gamma} \rho(D w)^{\gamma})$, \eqref{e:dt10} implies
\begin{equation} \label{e:dt10.1}
\int_{\Omega} \rho(D( \eta w))^{\gamma} \le c_{\gamma}(1 + \beta^{\gamma}) \left[ \int_{\Omega} w^{\gamma} \rho(D \eta)^{\gamma} + \int_{\Omega} (\eta w)^{\gamma} \left(
\frac{\rho_{*}(\vec{F})^{\gamma^{\prime}}}{\bar{k}^{\gamma}} + \frac{f}{\bar{k}^{\gamma-1}} \right) \right].
\end{equation}

Our next goal is to deal with the term $\int (\eta w)^{\gamma} \frac{\rho_{*} (\vec{F})^{\gamma^{\prime}}}{\bar{k}^{\gamma}}$. We now outline how we will do this: We  first use Holder's inequality to make the $L^{q}$ norm of $\vec{F}$ appear. To this end, we will introduce the parameter $\alpha_{1} = \alpha(q,\gamma) = \frac{q}{q-\gamma^{\prime}}$ which is equivalent to $\alpha_{1}^{\prime} = \frac{q}{\gamma^{\prime}}$. This is precisely where the requirement $f \in L^{\frac{q}{\gamma^{\prime}}}$ comes from.

Next, by choosing $\bar{k}$ appropriately, we will make the term with $\rho(\vec{F}) + f$ be absorbed into a $1$. At this point, the remaining term with $\eta w$ will have a strange power. So, we use interpolation, and the fact that $ \gamma < \gamma \alpha_{1} < \frac{n \gamma}{n- \gamma}$, to re-write the strange power as a linear combination of the $L^{\gamma}$ and $L^{\gamma^{*}}$ norms of $\eta w$, where $\gamma^{*}$ is the Sobolev conjugate of $\gamma$.  The necessary upper-bound on $\alpha_{1}$ is satisfied so long as $q > \frac{n}{\gamma-1}$. 

By making the coefficient of the $L^{\gamma}$ norm of $\eta w$ larger, we can choose the $L^{\gamma^{*}}$ norm of $\eta w$ to be arbitrarily small. This is necessary to apply the Gagliardo-Nirenberg-Sobolev inequality to turn this latter norm into an estimate on the $L^{\gamma}$ norm of $\rho(D (\eta w))$ which can then be absorbed into the left hand side. Applying Young's inequality makes this weighted-linear combination of norms appear, at which point we will choose $\alpha_{2} = \alpha(n,q,\gamma) = \theta_{1}^{-1}$, where $\theta_{1}$ is the interpolation power. This conveniently makes all exponents outside of integrals disappear, allowing the desired simplifications to all occur.

We now begin the process outlined above by applying Holder's inequality,
\begin{align} \nonumber
\int_{\Omega}(\eta w)^{\gamma} \left(\frac{\rho_{*}(\vec{F})^{\gamma^{\prime}}}{k^{\gamma}} + \frac{f}{k^{\gamma-1}} \right)& \le \| (\eta \omega)^{\gamma}\|_{L^{\alpha_{1}}(\Omega)}  \left\| \frac{\rho_{*}(\vec{F})^{\gamma^{\prime}}}{k^{\gamma}} +\frac{ f }{k^{\gamma-1}} \right\|_{L^{\alpha_{1}^{\prime}}(\Omega)} \\
\label{e:dt11} & \le C(\gamma,q) \| \eta w \|_{L^{\gamma \alpha_{1}}(\Omega)}^{\gamma}
\end{align}
where $\alpha_{1}^{\prime} = \frac{q}{\gamma^{\prime}}$ is as above, and $\bar{k}$ is chosen so that $\bar{k} = k_{1} + k_{2}$ where $k_{1}^{\gamma} = \|\rho_{*}(\vec{F})\|_{L^{q}(\Omega)}^{\gamma^{\prime}}$ and $k_{2}^{\gamma-1} = \|f\|_{L^{\frac{q}{\gamma^{\prime}}}(\Omega)}$. If $\vec{F},f \equiv 0$ choose $\bar{k} > 0$ arbitrary, and you can later send $\bar{k}$ to zero. Next, we define $\theta_{1} = \theta_{1}(q, n,\gamma) \in (0,1)$ so that $\frac{1}{\gamma \alpha_{1}} = \frac{\theta_{1}}{\gamma} + \frac{(1- \theta_{1})(n-\gamma)}{n \gamma}$. Note that if $\alpha_{2} = \alpha(q,n,\gamma) = \theta_{1}^{-1}$ then $\alpha_{2}^{\prime} = (1-\theta_{1})^{-1}$.  Riesz-Thorin interpolation applied to \eqref{e:dt11}, Young's inequality, and the Gagliardo-Nirenberg-Sobolev inequality consecutively ensure
\begin{align} \nonumber
\int_{\Omega}(\eta w)^{\gamma} \frac{\rho_{*}(\vec{F})^{\gamma^{\prime}}}{k^{\gamma}}& \le \left( \| \eta w\|_{L^{\gamma}(\Omega)}^{\theta_{1}} \| \eta w\|_{L^{\gamma^{*}}(\Omega)}^{1-\theta}\right)^{\gamma} \\
& \nonumber \le  \frac{1}{\epsilon^{\alpha_{2}} \alpha_{2}} \|\eta w\|_{L^{\gamma}(\Omega)}^{\gamma \theta_{1} \alpha_{2}} + \frac{\epsilon^{\alpha_{2}^{\prime}}}{\alpha_{2}^{\prime} } \|\eta w\|_{L^{\gamma^{*}}(\Omega)}^{\gamma(1 - \theta_{1}) \alpha_{2}^{\prime}} \\
& \nonumber =  \frac{1}{\epsilon^{\alpha_{2}} \alpha_{2}} \|\eta w\|_{L^{\gamma}(\Omega)}^{\gamma} + \frac{\epsilon^{\alpha_{2}^{\prime}}}{\alpha_{2}^{\prime} } \|\eta w\|_{L^{\gamma^{*}}(\Omega)}^{\gamma} \\
\label{e:dt11.1} & \le \epsilon^{- \alpha_{2}} \alpha_{2}^{-1} \| \eta w \|_{L^{\gamma}(\Omega)}^{\gamma} + C(q,n,\gamma,\nu) \epsilon^{\alpha_{2}^{\prime}}  \| \rho(D(\eta w)) \|_{L^{\gamma}(\Omega)}^{\gamma}.
\end{align}

See Remark \ref{r:sobolev} for our non-standard application of Gagliardo-Nirenberg-Sobolev inequality in \eqref{e:dt11.1}. Next, we want to plug \eqref{e:dt11.1} into \eqref{e:dt10.1} and subtract over the $\rho(D(\eta w))$ term, so we choose $\epsilon$ so that the coefficient of $\| \rho(D(\eta w))\|_{L^{\gamma}(\Omega)}^{\gamma}$ is $1/2$. That is, choose $\epsilon = \epsilon(q,n,\gamma,\nu) >0$ by 
$$
c_{\gamma}(1+\beta^{\gamma}) C(q,n,\gamma,\nu) \epsilon^{\alpha^{\prime}_{2}} = \frac{1}{2}. 
$$
Then the choice of $\epsilon$, \eqref{e:dt11.1}, and \eqref{e:dt10.1} imply
\begin{align} \nonumber
\int_{\Omega} \rho(D(\eta w))^{\gamma} & \le c_{\gamma}(1 + \beta^{\gamma}) \left[ \int_{\Omega} w^{\gamma} \rho(D \eta)^{\gamma} + C_{q,n,\gamma,\nu} (1 + \beta^{\gamma})^{\alpha_{2}-1} \int_{\Omega} w^{\gamma} \eta^{\gamma} \right] \\
 \label{e:dt12} & \le  C_{q,n,\gamma,\nu} (1 + \beta^{\gamma})^{\alpha_{2}} \left[ \int_{\Omega} w^{\gamma} (\rho(D \eta)^{\gamma} + \eta^{\gamma}) \right],
\end{align}
where as always, $\alpha_{2} = \alpha(n,q,\gamma) > 0$. Finally, the Gagliardo-Nirenberg-Sobolev inequality applied to \eqref{e:dt12} implies
\begin{equation} \label{e:dt13}
\| \eta w \|_{\gamma \chi} \le C_{q,n,\gamma,\nu}(1 + \beta^{\gamma})^{\frac{\alpha_{2}}{\gamma}} \left[ \int_{\Omega} w^{\gamma}(\rho(D \eta)^{\gamma} + \eta^{\gamma}) \right]^{\frac{1}{\gamma}},
\end{equation}
where $\chi = \frac{n}{n-\gamma}$. Choose the cut-off function $\eta$ so that with $0 < r < R <1$ and some $B_{R}(x) \subset \subset \Omega$, $\eta \in C^{1}_{0}(B_{R})$ and
\begin{equation} \label{e:dt13.1}
\eta \equiv 1 \text{ in } B_{r} \quad \text{and} \quad |D \eta| \le \frac{2}{R-r}.
\end{equation}

Then \eqref{e:dt13} guarantees
\begin{equation*} 
\left(\int_{B_{r}(x)} w^{\gamma \chi}\right)^{\frac{1}{\chi}} \le C_{n,\gamma,q, \nu, \Lambda} \frac{(1+ \beta^{\gamma-1})^{\alpha_{2}}}{(R-r)^{\gamma}} \int_{B_{R}} w^{\gamma}
\end{equation*}
where our constant gained a dependence on $\Lambda$, which arises by applying \eqref{e:dt13.1} in the form $\rho(D \eta) \le \Lambda |D \eta| \le  2\Lambda(R-r)^{-1}$. Recall the definition of $w$ and use $v \le \bar{u}$ to discover
\begin{equation} \label{e:dt14} 
\left(\int_{B_{r}(x)} v^{(\beta + \gamma)\chi}  \right)^{\frac{1}{\chi}} \le C_{n,\gamma,q, \nu, \Lambda } \frac{(1+ \beta^{\gamma-1})^{\alpha_{2}}}{(R-r)^{\gamma}} \int_{B_{R}} \bar{u}^{\beta + \gamma}.
\end{equation}
 Taking $m \uparrow \infty$ in \eqref{e:dt14} yields
\begin{equation} \label{e:dt15}
\| \bar{u} \|_{L^{(\beta + \gamma)\chi}B_{r}(x)}  \le \left( C_{n,\gamma,q, \nu , \Lambda } \frac{(1+ \beta^{\gamma-1})^{\alpha_{2}}}{(R-r)^{\gamma}}\right)^{\frac{1}{\beta + \gamma}} \| \bar{u}\|_{L^{\beta + \gamma}(B_{R}(x))},
\end{equation}
so long as the right hand side is finite. Note the constant on the righthand side is independent of $\beta$, suggesting we iterate. We begin with $\beta = \beta_{0} = 0$. For $k \ge 1$ define $\beta_{k} = \gamma( \chi^{k-1}-1)$. For $k \ge 0$ we consider $r_{k} = r + \frac{R-r}{2^{k+1}}$. Note, $(\beta_{k} + \gamma)\chi = \gamma \chi^{k} = (\beta_{k+1}+\gamma)$ and $r_{k-1} - r_{k} = 2^{-(k+1)} (R-r)$. Writing $C = C_{n,\gamma,q, \nu, \Lambda}$, for $k \ge 1$, with these choices \eqref{e:dt15} reads
\begin{align*} 
\| \bar{u}\|_{\displaystyle L^{\beta_{k} + \gamma}(B_{r_{k}}(x))} & \le \left( C \frac{ (1 + \beta_{k}^{\gamma-1})^{\alpha_{2}}}{(R-r)^{\gamma} } \right)^{\frac{1}{\beta_{k}+\gamma}} 2^{\frac{k+1}{\gamma \chi^{k-1}}} \| \bar{u} \|_{\displaystyle  L^{\beta_{k-1} + \gamma}(B_{r_{k-1}}(x))}.
\end{align*}
Since $\beta_{k} = \gamma(\chi^{k-1}-1)$ we observe that for $k \ge 2$,
\begin{align*} 
\left( \frac{ (1 + \beta_{k}^{\gamma-1})^{\alpha_{2}}}{(R-r)^{\gamma} } \right)^{\frac{1}{\beta_{k}+\gamma}} & \le (R-r)^{- \chi^{k}} (C (1+\gamma \chi))^{\frac{k (\gamma-1) \alpha_{2}}{\gamma \chi^{k-1} -1}} \\
& \le (R-r)^{-\chi^{k}} C^{ \frac{k}{\chi^{k}}}.
\end{align*}
After iterating it follows that for all $k \ge 2$,
$$
\| \bar{u}\|_{\displaystyle L^{\beta_{k} + \gamma}(B_{r_{k}}(x))} \le (R-r)^{-\sum_{i=0}^{\infty} \chi^{i}}C^{ \sum_{i=0}^{\infty} \frac{i}{\chi^{i-1}}} \| \bar{u}\|_{L^{\beta_{0}+\gamma }(B_{r_{1}}(x))},
$$
where still $C= C(n,\gamma, q, \nu, \Lambda)$. Since $\chi > 1$, $u \in W^{1,\gamma}(\Omega)$, and $\beta_{0} = 0$, the right hand side is seen to be finite, and is independent of $k$. Taking $k \to \infty$ on the left hand side, and recalling $\bar{u} = u^{+} + \bar{k}$ yields 
$$
 \| u^{+} \|_{L^{\infty}(B_{r}(x))} \le C (R-r)^{-\frac{n}{\gamma}}\left[ \|u^{+}\|_{L^{\gamma}(B_{R})} + \bar{k} \right].
$$
Recalling $\bar{k}= \| \rho_{*}(\vec{F})\|_{q}^{\frac{\gamma^{\prime}}{\gamma}} + \|f\|_{\frac{q}{\gamma^{\prime}}}^{\frac{1}{\gamma-1}}$ and noting $\frac{1}{1- \chi^{-1}} = \frac{n}{\gamma}$ yields the result for the case $p \ge \gamma$. A classical scaling argument covers the case $0 < p < \gamma$. See, for instance, \cite[p. 75]{han2011elliptic}. 
\end{proof}

\begin{theorem} \label{t:inf}
Let $\rho : \Omega \times \rn \setminus \{0\}$ satisfy \eqref{e:meas}, \eqref{e:p1h}, \eqref{e:cx}, and \eqref{e:eb}. Suppose $u \in W^{1,\gamma}(\Omega)$ is a supersolution to \eqref{e:dt1} in $\Omega$ and $\vec{F},f$, and $q$ are as in \eqref{e:Ff}. If $B_{2 R} \subset \Omega$ and $0 < p < \frac{(\gamma-1) n}{n- \gamma}$, then for any $0 < \theta < \tau < 1$ there exists $C = C(n,\gamma,q,p,\nu,\Lambda,\theta,\tau)$ and $\delta = 1 - \frac{n}{q(\gamma-1)} > 0$ so that
$$
\inf_{B_{\theta R}} u^{+} + R^{\delta} \|\rho_{*}(\vec{F})\|_{L^{q}(B_{R})}^{\frac{1}{\gamma-1}} + R^{\gamma^{\prime}\delta} \| f\|_{L^{\frac{q}{\gamma^{\prime}}}(B_{R})}^{\frac{1}{\gamma-1}} \ge C R^{-\frac{n}{p}}\| u^{+}\|_{L^{p}(B_{\tau R})}.
$$
\end{theorem}

Two corollaries of Theorem \ref{t:sup} and Theorem \ref{t:inf} are the strong maximum principle and the weak-Harnack inequality (Theorems \ref{t:mp} and \ref{t:iwh}), which we will state and prove before proving Theorem \ref{t:iwh}.

\begin{theorem} \label{t:mp}
Let $\rho : \Omega \times \rn \setminus \{0\}$ satisfy \eqref{e:meas}, \eqref{e:p1h}, \eqref{e:cx}, and \eqref{e:eb}. Suppose $u \in W^{1,\gamma}(\Omega)$ is a super solution to \eqref{e:dt1} in $\Omega$, and $\vec{F}, f \equiv 0$. Then, if for some ball $B \subset \subset \Omega$
$$
\sup_{B} u = \sup_{\Omega} u \ge 0,
$$
then the function $u$ must be constant in $\Omega$.
\end{theorem}

\begin{proof}
Suppose $B^{\prime} \subset \subset \Omega$ is so that $\sup_{B^{\prime}} u = \sup_{\Omega} u$. If necessary there is a much smaller ball $B_{R} \subset \subset \Omega$ so that $\sup_{B_{R/3}} u = \sup_{B_{R/3}} \Omega$. Without loss of generality $B_{2R} \subset \subset \Omega$. Let $M = \sup_{\Omega} u$ and applying Theorem \ref{t:iwh} to the non-negative supersolution $M-u$ it follows that 
$$
R^{-n} \|M-u\|_{L^{1}(B_{2R/3})} \le C \inf_{B_{R/3}} (M-u) = 0. 
$$
Hence, $u \equiv M$ on $B_{2R/3}$, which implies the theorem.
\end{proof}

\begin{theorem} \label{t:iwh}
Suppose $\rho : \Omega \times \rn \setminus \{0\} \to (0, \infty)$ satisfies \eqref{e:meas}, \eqref{e:p1h}, \eqref{e:cx}, and \eqref{e:eb}. If $1 < \gamma < n$ and $u \in W^{1,\gamma}(\Omega)$ is a nonnegative solution to \eqref{e:dt1} in $\Omega$ for some $\vec{F}$, $f$, and $q$ are as in \eqref{e:Ff} and some $\overline{B_{3R}} \subset \Omega$, then there exists some $C = C(n, \gamma, q, \nu, \Lambda)$ and $\delta = 1 - \frac{n}{q(\gamma-1)} > 0$ so that
\begin{equation} \label{e:iwh}
\sup_{B_{R}} u \le C_{n, \gamma,\nu, \Lambda,q} \left[ \inf_{B_{2R}} u+ R^{\delta} \|\rho_{*}(\vec{F})\|_{L^{q}(B_{3R})}^{\frac{1}{\gamma-1}} + R^{\gamma^{\prime}\delta} \| f\|_{L^{\frac{q}{\gamma^{\prime}}}(B_{3R})}^{\frac{1}{\gamma-1}} \right] .
\end{equation}
\end{theorem}

The preceding Theorem follows readily from Theorem \ref{t:sup} and Theorem \ref{t:inf} by choosing, for instance, $p =  \frac{(\gamma-1)n}{2(n-\gamma)}$ in both theorems. In the case that $\vec{F}, f \equiv 0$ this recovers an inequality of the same form as the classical Harnack inequality.

\begin{proof}(Of Theorem \ref{t:inf}.)

Let $k > 0$ to be chosen later, $\bar{u} = u^{+} + k$ , $v = \min\{ \bar{u}, m\}$. Consider the test function $\varphi = \eta^{\gamma} v^{-\beta} \bar{u}$ for $\beta \ge \beta_{p} > 1$, where $\beta_{p} = \gamma - \frac{p}{\chi}$, $\chi = \frac{n}{n - \gamma}$.

We now proceed in a similar fashion to the proof of Theorem \ref{t:sup}. Observe,
\begin{align}
\nonumber \big \langle \rho(Du)^{\gamma-1} (D \rho)(Du),& -\beta v^{-\beta-1} \bar{u} \eta^{\gamma} D v + \eta^{\gamma} v^{-\beta} D \bar{u} \big \rangle  \\
\nonumber & = -\beta \rho(Dv)^{\gamma} v^{-\beta} \eta^{\gamma} + \rho(D\bar{u})^{\gamma} \eta^{\gamma} v^{-\beta} \\
\label{e:lb0} & = (1- \beta) \rho(Dv)^{\gamma} v^{-\beta} \eta^{\gamma}.
\end{align}
We note that $(\beta-1) \ge 1 - \beta_{p}$; this observation will simplify our computations later as, when we iterate $\beta$, we may now keep constants independent of $\beta$. They will instead depend on $\beta_{p}$, which depends on $n, \gamma, p$. 

Analogous to \eqref{e:dt3} we compute

\begin{equation} \label{e:lb1}
\Langle \rho(Du)^{\gamma-1}(D \rho)(Du), \gamma \eta^{\gamma-1} v^{-\beta} \bar{u} D \eta \Rangle \le \gamma v^{-\beta} \left[ \epsilon^{\gamma^{\prime}} \frac{ \rho(D \bar{u})^{\gamma} \eta^{\gamma}}{\gamma^{\prime}} + \frac{\rho(D \eta)^{\gamma} \bar{u}^{\gamma}}{\epsilon^{\gamma} \gamma}\right].
\end{equation}
Choosing $\epsilon$ so that $(\gamma-1)\epsilon^{\gamma^{\prime}} = -(1- \beta)/2 >(1 - \gamma)/2 > 0$ yields
\begin{align} 
\nonumber \bigg \langle \rho(Du)^{\gamma-1}(D \rho)(Du), & \gamma \eta^{\gamma-1} v^{-\beta} \bar{u} D \eta \bigg \rangle \\ 
\label{e:lb2} & \le  \frac{- (1-\beta) }{2} v^{-\beta}\eta^{\gamma}  \rho(D v)^{\gamma} + c_{n,\gamma,p} v^{-\beta} \rho(D \eta)^{\gamma} \bar{u}^{\gamma}.
\end{align}
Next we treat the $\vec{F}$ term. Proceeding as in \eqref{e:dt5} we discover
\begin{align} 
\nonumber \big\langle \vec{F}, &  -\beta v^{-\beta-1} \bar{u} \eta^{\gamma} Dv + \eta^{\gamma} v^{\beta} D \bar{u} \big\rangle  \ge - (\beta-1)  \eta^{\gamma}v^{-\beta} \left[  \left(\rho_{*}(\vec{F}) \right) \left(  \rho(Dv) \right) \right] \\
\nonumber & \ge (1- \beta) \left[ \frac{ \eta^{\gamma} v^{-\beta} \rho_{*}(\vec{F})^{\gamma^{\prime}}}{\epsilon^{\gamma^{\prime}} \gamma^{\prime}} + \epsilon^{\gamma} \frac{ \eta^{\gamma} v^{-\beta} \rho(Dv)^{\gamma}}{\gamma} \right] \\
\label{e:lb3} & \ge  \frac{(1-\beta)}{4} \eta^{\gamma} v^{-\beta} \rho(Dv)^{\gamma} -  c_{n,\gamma,p}  \eta^{\gamma} v^{-\beta} \rho_{*}(\vec{F})^{\gamma^{\prime}}
\end{align}
where we chose $\epsilon > 0$ so that $\gamma^{-1} \epsilon^{\gamma}= 1/4$. To bound the final piece, we compute

\begin{align}
\nonumber \Langle \vec{F}, \gamma \eta^{\gamma-1} v^{-\beta} \bar{u} D \eta \Rangle & \ge - \gamma \rho_{*}(\vec{F}) \eta^{\gamma-1} v^{-\beta} \bar{u} \rho(D \eta) \\
\nonumber & \ge -  \gamma v^{-\beta} \left[ \frac{ \rho_{*}(\vec{F})^{\gamma^{\prime}} \eta^{\gamma} }{\gamma^{\prime}} + \frac{ \bar{u}^{\gamma} \rho(D \eta)^{\gamma}}{\gamma} \right] \\
\label{e:lb4} & = - (\gamma-1) \rho_{*}(\vec{F})^{\gamma^{\prime}} \eta^{\gamma} v^{-\beta} - \bar{u}^{\gamma} \rho(D \eta)^{\gamma} v^{-\beta}.
\end{align}

Using that $u$ is a supersolution and plugging \eqref{e:lb0}, \eqref{e:lb2}, \eqref{e:lb3}, and \eqref{e:lb4} into \eqref{e:dt1} achieves
\begin{align*}
& \frac{-(\beta-1)}{4}  \int_{\Omega} \rho(Dv)^{\gamma} v^{-\beta} \eta^{\gamma}\\
 & \ge - c_{n,\gamma,p} \int_{\Omega} v^{-\beta} \rho(D \eta)^{\gamma} \bar{u}^{\gamma} - c_{n,\gamma,p} \int_{\Omega} \eta^{\gamma} v^{-\beta} \rho_{*}(\vec{F})^{\gamma^{\prime}} - \int_{\Omega} \eta^{\gamma} v^{-\beta+1} f
\end{align*}
or after recalling $\bar{u} \ge k$, $\beta-1 \ge \beta_{p} -1 >0$ and multiplying by $-1$,
\begin{align} \label{e:lb5}
\nonumber \int_{\Omega} \rho(Dv)^{\gamma} v^{-\beta} \eta^{\gamma} &\le c_{n,\gamma,p} \bigg[ \int_{\Omega} v^{-\beta} \bar{u}^{\gamma} \rho(D \eta)^{\gamma} + \\
& + \int_{\Omega} \eta^{\gamma} v^{-\beta} \bar{u}^{\gamma} \frac{\rho_{*}(\vec{F})^{\gamma^{\prime}}}{k^{\gamma}} + \int_{\Omega} \eta^{\gamma} v^{-\beta} \bar{u}^{\gamma} \frac{f}{k^{\gamma-1}} \bigg].
\end{align}

Note the striking similarity to \eqref{e:dt8}. The only difference being the benefit that \eqref{e:lb5} has no constants depending on $\beta$. Hence, we follow the process done in the proof of Theorem \ref{t:sup} and choose 
$$
k = \| \rho_{*}(\vec{F})\|_{L^{q}}^{\frac{1}{\gamma-1}} + \| f \|_{L^{\frac{q}{\gamma^{\prime}}}}^{\frac{1}{\gamma-1}}.
$$ 
Recalling $0 < \theta < \tau < 1$, setting $w = v^{- \beta/\gamma} \bar{u}$ and proceeding as in the proof of Theorem \ref{t:sup} leads to
\begin{equation} \label{e:lb6}
\left( \int_{B_{R \theta}} \bar{u}^{(\gamma-\beta)\chi} \right)^{\frac{1}{\chi}} \le C_{n, \gamma,\nu, \Lambda, p,q, R \theta, R\tau} \int_{B_{R \tau}} \bar{u}^{\gamma-\beta}.
\end{equation}
The dependence on $R \theta, R\tau$ comes from the magnitude of the derivative of the cut-off function. Moreover, in \eqref{e:lb6}, the dependence on $p$ that is not present in \eqref{e:dt13} is due to having assumed $\beta \ge \beta_{p}$. To improve readability, we write $C = C_{n,\gamma,\nu,\Lambda,p,q,R \theta,R\tau}$ and suppose $R=1$ through the end of \eqref{e:lb8.4}. 
Now, whenever $\gamma- \beta < 0$ \eqref{e:lb6} ensures
\begin{equation} \label{e:lb7}
C  \| \bar{u}\|_{L^{\gamma-\beta}(B_{\tau})} \le \| \bar{u}\|_{L^{(\gamma-\beta)\chi}(B_{\theta})}.
\end{equation} 
Since $\inf v^{+} = \lim_{p \to - \infty} \| v\|_{L^{p}}$, iterating as in Theorem \ref{t:sup}, for any $\beta > \gamma$ \eqref{e:lb7} implies,
\begin{equation} \label{e:lb8}
C\| \bar{u}\|_{L^{\gamma-\beta}(B_{\tau})} \le \inf_{B_{\theta}} \bar{u}.
\end{equation}

On the other hand, whenever $0 < \beta < \gamma-1$, \eqref{e:lb6} guarantees
\begin{equation} \label{e:lb8.1}
\| \bar{u}\|_{L^{(\gamma-\beta)\chi}(B_{\theta})} \le C  \|\bar{u}\|_{L^{\gamma-\beta}(B_{\tau})}.
\end{equation}
Recalling $- \beta < -1$, we can iterate \eqref{e:lb8.1} finitely many times depending on $p,p_{0}$ so that when $0 < p_{0} < p < (\gamma-1)\chi$, choosing $\gamma-\beta = p_{0}$ and finitely many iterations of \eqref{e:lb8.1} ensures
\begin{equation} \label{e:lb8.2}
\| \bar{u}\|_{L^{p}(B_{\theta})} \le C \cdot C(p_{0}) \| \bar{u}\|_{L^{p_{0}}(B_{\tau})}.
\end{equation}

We claim the proof is complete once we show that there exists $p_{0} > 0$ so that
\begin{equation} \label{e:lb8.3}
\int_{B_{\tau}} \bar{u}^{-p_{0}} \int_{B_{\tau}} \bar{u}^{p_{0}} \le C.
\end{equation}

Indeed, choosing $\beta = \gamma + p_{0}$ in \eqref{e:lb8} combined with \eqref{e:lb8.2} and \eqref{e:lb8.3} yields
\begin{align} \label{e:lb8.4}
\inf_{B_{1}} \bar{u} & \ge C \cdot C(p_{0}) \| \bar{u}\|_{L^{-p_{0}}} \\
\nonumber & =C \cdot C(p_{0}) \left( \| \bar{u}\|_{L^{-p_{0}}} \| \bar{u}\|_{L^{p_{0}}}^{-1} \right) \| \bar{u}\|_{L^{p_{0}}} \\
\nonumber & \ge C \cdot C(p_{0}) \| \bar{u}\|_{L^{p}}.
\end{align}
It happens that $p_{0} = p_{0}(n,\gamma,\nu,\Lambda)$ so that the constant $C(p_{0})$ can be absorbed into our universal constant. Recalling that $\bar{u} = u^{+} + k$ and using scaling, this says
$$
\inf_{B_{R \theta}} u^{+} + R^{\delta} \|\rho_{*}(\vec{F})\|_{L^{q}(B_{R})}^{\frac{1}{\gamma-1}} + R^{\gamma^{\prime} \delta} \| f\|_{L^{\frac{q}{\gamma^{\prime}}}(B_{R})}^{\frac{1}{\gamma-1}} \ge C_{n, \gamma,  \nu, \Lambda, q, p, \theta, \tau} \| u^{+}\|_{L^{p}(B_{R \tau})},
$$
as desired.

Hence, it only remains to show \eqref{e:lb8.3}. We consider the test function $\varphi = \frac{\eta^{\gamma}}{\bar{u}^{\gamma-1}}$. Note,
$$
D \varphi = \gamma \left( \frac{\eta}{\bar{u}} \right)^{\gamma-1} D \eta - (\gamma-1) \left( \frac{\eta}{\bar{u}} \right)^{\gamma} D \bar{u}.
$$
We estimate as usual

\begin{equation} \label{e:lb9}
\Langle \rho(D \bar{u})^{\gamma-1} (D \rho) (D \bar{u}), - (\gamma-1) \left( \frac{\eta}{\bar{u}} \right)^{\gamma} D \bar{u} \Rangle= - (\gamma-1) \rho\left( \frac{ \eta D \bar{u}}{\bar{u}}\right)^{\gamma} .
\end{equation}
and choosing $\epsilon > 0$ so that $(\gamma-1) \epsilon^{\gamma^{\prime}} = (\gamma-1)/2$.
\begin{align} \label{e:lb10}
\nonumber \Langle \rho(D \bar{u})^{\gamma-1} (D \rho)(D \bar{u}), \gamma \left( \frac{\eta}{\bar{u}} \right)^{\gamma-1} D \eta \Rangle & \le   \gamma \left[\frac{ \epsilon^{\gamma^{\prime}}}{\gamma^{\prime}} \rho \left( \frac{ \eta D \bar{u}}{\bar{u}} \right)^{\gamma} + \frac{ \rho(D \eta)^{\gamma}}{\epsilon^{\gamma} \gamma} \right] \\
& \le \frac{\gamma-1}{2} \rho \left( \frac{ \eta D \bar{u}}{\bar{u}} \right)^{\gamma} + c_{\gamma} \rho(D \eta)^{\gamma}.
\end{align}

On the other hand, since $\bar{u} \ge k$
\begin{align}
\nonumber \Langle \vec{F}, \gamma \left( \frac{\eta}{\bar{u}} \right)^{\gamma-1} D \eta \Rangle & \ge - \gamma \rho_{*}(\vec{F})^{\gamma^{\prime}} \left( \frac{ \eta}{\bar{u}} \right)^{\gamma-1} \rho(D \eta) \\
\label{e:lb11}& \ge - \gamma  \left[ \frac{\rho_{*}(\vec{F})^{\gamma^{\prime}} \eta^{\gamma}}{k^{-\gamma} \gamma^{\prime}} + \frac{ \rho(D \eta)^{\gamma}}{\gamma} \right]
\end{align}
and choosing $\epsilon > 0$ so that $(\gamma^{\prime})^{-1} \epsilon^{\gamma} = 1/4$ guarantees
\begin{align}
\nonumber \langle \vec{F}, -(\gamma-1) \left( \frac{\eta}{\bar{u}} \right)^{\gamma} D \bar{u} \rangle & \ge - (\gamma-1) \left( \frac{\eta}{\bar{u}} \right)^{\gamma} \left[ \frac{ \epsilon^{\gamma} \rho(D \bar{u})^{\gamma}}{\gamma} + \frac{\rho_{*}(\vec{F})^{\gamma^{\prime}}}{\epsilon^{\gamma^{\prime}} \gamma^{\prime}} \right]  \\
\label{e:lb12}& \ge - \frac{(\gamma-1)}{4} \rho \left( \frac{ \eta D \bar{u}}{\bar{u}} \right)^{\gamma} - c_{\gamma} \eta^{\gamma} k^{-\gamma} \rho_{*}(\vec{F})^{\gamma^{\prime}}.
\end{align}
Finally, note
\begin{equation} \label{e:lb13}
\int f \eta^{\gamma} \bar{u}^{-(\gamma-1)} \ge - k^{-(\gamma-1)} \| f \|_{\frac{n}{\gamma}} \| \eta \|_{L^{\frac{n \gamma}{n-\gamma}}}^{\gamma} \ge - c_{\gamma,n,\nu} k^{-(\gamma-1)} \| f\|_{\frac{n}{\gamma}} \|D \eta \|_{L^{\gamma}}^{\gamma}
\end{equation}

Combining \eqref{e:lb9} - \eqref{e:lb13}, with the fact that $u$ is a supersolution yields
\begin{align}
\nonumber - \int_{\Omega} &  \eta^{\gamma} \rho \left(  D (\log \bar{u}) \right)^{\gamma} \ge - c_{n,\gamma, \nu} \left[ \int \rho( D \eta)^{\gamma} \left( 1 + \frac{\|f\|_{\frac{n}{\gamma}}}{k^{\gamma-1}} \right) + \int \frac{ \rho_{*}(\vec{F})^{\gamma^{\prime}}}{k^{\gamma}} \eta^{\gamma} \right] \\
& \ge  - c_{n,\gamma, \nu} \left[ \int \rho(D \eta)^{\gamma} \left( 1 + k^{-(\gamma-1)} \|f\|_{\frac{n}{\gamma}} + k^{-\gamma} \|\rho_{*}(\vec{F})\|_{L^{\frac{n}{\gamma-1}}}^{\gamma^{\prime}} \right) \right] ,
\end{align}
where the 2nd inequality follows from the first by Holder and Sobolev embedding similar to \eqref{e:lb13}. 
Choosing $k = \| f \|_{\frac{n}{\gamma}}^{\frac{1}{\gamma-1}} + \|\rho_{*}(\vec{F})\|_{L^{\gamma^{\prime}}}^{\frac{1}{\gamma-1}}$ this can be written
\begin{equation}  \label{e:lb14}
\int_{\Omega} \eta^{\gamma} \rho(  D ( \log \bar{u}))^{\gamma} \le c_{n, \gamma, \nu} \int \rho(D \eta)^{\gamma}.
\end{equation}
Now, for all $B_{2r}^{\prime} \subset B_{2R}$, we can choosing $\eta \equiv 1$ on $B_{r}^{\prime}$ and $\eta \equiv 0$ on $\Omega \setminus B_{2r}^{\prime}$ with $|D \eta| \le \frac{2}{r}$, \eqref{e:lb14} says

\begin{equation*}
\int_{B_{r}^{\prime}} |D ( \log \bar{u})|^{\gamma} \le c_{n,\gamma,\nu} \int_{B_{r}^{\prime}} \rho(D(\log \bar{u}))^{\gamma} \le c_{n,\gamma,\nu, \Lambda} r^{n - \gamma}
\end{equation*}
or taking the $\gamma$-root and multiplying by $r^{1 - \frac{n}{\gamma}}$ that is
\begin{equation} \label{e:lb14.1}
r^{1-\frac{n}{\gamma}} \left( \int_{B_{r}^{\prime}} |D(\log \bar{u})|^{\gamma} \right)^{\frac{1}{\gamma}} \le C_{n,\gamma,\nu,\Lambda}. 
\end{equation}

Noticing $(r^{-n})^{\frac{1}{\gamma^{*}}} = r^{1- \frac{n}{\gamma}}$, we consecutively apply Jensen's inequality, the Poincare inquality in a ball, and \eqref{e:lb14.1} to achieve
\begin{align} \label{e:lb15}
\frac{1}{r^{n}} \int_{B_{r}^{\prime}} |\log \bar{u} - (\log \bar{u})_{B_{r}^{\prime}}|  &\le C \left( \frac{1}{r^{n}} \int_{B_{r}^{\prime}} |\log \bar{u}  - (\log \bar{u})_{B_{r}^{\prime}}|^{\gamma^{*}} \right)^{\frac{1}{\gamma^{*}}} \\
\nonumber & \le C_{n, \gamma, \nu,\Lambda}.
\end{align}
Since this holds uniformly for all $B_{2r}^{\prime} \subset B_{2R}$, and hence \eqref{e:lb15} holds for all $B_{r}^{\prime} \subset B_{R}$, this ensures $\log \bar{u} \in BMO(B_{R})$ and consequently, by the John-Nirenberg lemma there exists $0< p_{0}$ depending only on $n$ and the constant in \eqref{e:lb15} so that
\begin{equation} \label{e:lb16}
\sup_{B\subset B_{3}} \frac{1}{|B|} \int_{B} e^{p_{0} |\log \bar{u} - (\log \bar{u})_{B}|} < \infty.
\end{equation}
Finally, since $e^{|a-b|} \ge 1$ we note that $e^{p_{1}|a-b|} \le e^{p_{0}|a-b|}$ if $p_{1} \le p_{0}$. In particular, without loss of generality, suppose $p_{0} < \frac{(\gamma-1)n}{n-\gamma}$. But then, \eqref{e:lb16} implies \eqref{e:lb8.3}.
\end{proof}

\begin{theorem} [Improvement of Oscillation] \label{t:ioo}
Suppose $u, \rho, \gamma, \vec{F},f$, and $q$ are as in Theorem \ref{t:iwh}. If $\overline{B_{3R}} \subset \Omega$, then for all $0 < \theta < 1$,
$$
\osc_{B_{\theta R}} u \le C_{n,\gamma,q, \nu, \Lambda}  \theta^{\alpha} \left[ \osc_{B_{R}} u  +  \| \rho_{*}(\vec{F})\|^{\frac{1}{\gamma-1}}_{L^{q}(B_{2R})} + \|f\|_{L^{\frac{q}{\gamma^{\prime}}}(B_{2R})}^{\frac{1}{\gamma-1}} \right].
$$
\end{theorem}

\begin{proof}
For $0 < s < 2R$ let 
$$
M_{s} = \sup_{B_{s}} u \qquad \text{and} \qquad m_{s} = \inf_{B_{s}} u.
$$
Then $v \in \{M_{2R} - u, u - m_{2r}\}$ is a positive solution of \eqref{e:dt1} on $B_{2R}$.

Consider $p = 1$ and $\theta = 1/2 < \tau < 1$ in Theorem \ref{t:inf}. That is, for a positive solution $v$,
\begin{equation} \label{e:ioo1}
\inf_{B_{R/2}} v + G(R) \ge C R^{-n} \int v
\end{equation}
where 
$$
G(R) = R^{\delta} \| \rho_{*}(\vec{F})\|^{\frac{1}{\gamma-1}}_{L^{q}(B_{R})}   + R^{\gamma^{\prime} \delta} \| f\|_{L^{\frac{q}{\gamma^{\prime}}}(B_{2R})}^{\frac{1}{\gamma-1}}.
$$
Note, $\inf_{B_{R/2}} M_{2R} - u = M_{2R} - M_{R/2}$ and $\inf_{B_{R/2}} u - m_{2R} = m_{r/2} - m_{2r}$. Therefore, applying \eqref{e:ioo1} to $M_{2R} - u$ and $u - m_{2R}$ yields
\begin{align*}
M_{2R} - M_{R/2} + G(R) &\ge  C R^{-n} \int M_{2R} - u \\
m_{R/2} - m_{2R} + G(R) & \ge C R^{-n} \int u - m_{2R}.
\end{align*}
Adding yields
$$
(M_{2R} - m_{2R}) - (M_{R/2} - m_{R/2}) + 2 G(R)\ge C (M_{2R} - m_{2R} ),
$$
or equivalently,
$$
\osc_{B_{R/2}} u \le (1 - C) \osc_{B_{2R}} u + 2\left[ R^{\delta} \| \rho_{*}(\vec{F})\|^{\frac{1}{\gamma-1}}_{L^{q}(B_{R})}   + R^{\gamma^{\prime}\delta} \| f\|_{L^{\frac{q}{\gamma^{\prime}}}(B_{2R})}^{\frac{1}{\gamma-1}} \right].
$$
Lemma \ref{l:ibs} verifies the result by choosing the parameters from Lemma \ref{l:ibs} by $\tau = 1/4$,$\tilde \delta = (1-C)$, and $\mu (1 - \frac{n}{q(\gamma-1)}) > \alpha$. The latter can be done by making $\alpha$ smaller if necessary. Notably $\alpha = \alpha(1/4,\tilde \delta)$ so it has the expected dependencies.
\end{proof}

Holder regularity is a classic result of the improvement of oscillation in Theorem \ref{t:ioo}.
\begin{corollary}  \label{c:calpha}
Suppose $u, \rho, \gamma, \vec{F},f$ and $q$ are as in Theorem \ref{t:iwh}. Then $u \in C_{\loc}^{\alpha}(\Omega)$ for some $\alpha = \alpha(n, \gamma, \nu, \Lambda, q)$.
\end{corollary}

Last, we conclude with a Liouville-type theorem in the case that $f, \vec{F} \equiv 0$.

\begin{theorem} \label{t:liouville} [Liouville Theorem]
Let $\rho, \gamma,$ and $u$ be as in Theorem \ref{t:iwh}. Suppose additionally that $f, \vec{F} \equiv 0$. If $\Omega = \rn$ and $u$ is bounded from above or below, then $u$ is constant.
\end{theorem}

\begin{proof}
It suffices to assume $u \ge 0$ by replacing $u$ with $- u + \sup_{\rn} u$ ($u - \inf_{\rn} u$, resp.) when $u$ is bounded above (below, resp.). Since $f, \vec{F} \equiv 0$, and $\Omega = \rn$, Theorem \ref{t:iwh} implies $u$ is bounded above. Indeed, for $x \in \rn$, $u(x) \le \sup_{B_{|x|}} u \le C_{n,\gamma,\nu,\lambda, Q} \inf_{B_{2|x|}} u \le C u(0)$. In particular, $\|u\|_{L^{\infty}(\rn)} < \infty$. Now, iterating Theorem \ref{t:ioo} says there exists $0 < \theta < 1$ so that for any $R > 0$ and integer $k$,
$$
\osc_{B_{R}} u \le \theta^{k} \osc_{B_{2^{k} R}} \le  \theta^{k} \left( 2 \|u \|_{L^{\infty}(\rn)} \right).
$$
Taking $k$ and $R$ to infinity consecutively completes the proof.
\end{proof}

\bibliographystyle{alpha}
\bibdata{references}
\bibliography{references}
\end{document}